\documentclass[12pt]{amsart}

\usepackage[margin=1.5in]{geometry}

\usepackage{amsmath}
\usepackage{amsfonts}
\usepackage{amssymb}
\usepackage{amsthm}
\usepackage{enumerate}
\usepackage{graphicx}
\usepackage{hyperref}
\usepackage{enumitem}

\input xy
\xyoption{all}

\theoremstyle{plain}
\newtheorem{theorem}{Theorem}[section]
\newtheorem{corollary}[theorem]{Corollary}
\newtheorem{lemma}[theorem]{Lemma}
\newtheorem{proposition}[theorem]{Proposition}
\theoremstyle{definition}

\newtheorem{rem}[theorem]{Remark}

\DeclareMathOperator{\End}{End}

\DeclareMathOperator{\fix}{Fix}

\newcommand{\bb}[1]{\mathbb{#1}}
\newcommand{\stab}{\mathrm{Stab}}
\newcommand{\im}{\mathrm{im}}

\def\Z{{\mathbb{Z}}}
\def\Q{{\mathbb{Q}}}

\def\C{{\mathbb{C}}}

\makeatletter
\def\ps@pprintTitle{%
  \let\@oddhead\@empty
  \let\@evenhead\@empty
  \let\@oddfoot\@empty
  \let\@evenfoot\@oddfoot
}
\makeatother

\usepackage{marvosym}

\title{Smooth quotients of abelian varieties by finite groups}
\author{Robert Auffarth}
\address{R. Auffarth \\Departamento de Matem\'aticas, Facultad de
Ciencias, Universidad de Chile\\ Las Palmeras 3425, \~Nu\~noa, Santiago, Chile}
\email{rfauffar@uchile.cl}
\author{Giancarlo Lucchini Arteche}
\address{G. Lucchini Arteche \\Departamento de Matem\'aticas, Facultad de
Ciencias, Universidad de Chile\\ Las Palmeras 3425, \~Nu\~noa, Santiago, Chile}
\email{luco@uchile.cl}

\thanks{The first author was partially supported by Fondecyt Grant 11180965. The second author was partially supported by Fondecyt Grant 11170016 and PAI Grant 79170034.}


\begin{document}

\begin{abstract}
We give a complete classification of smooth quotients of abelian varieties by finite groups that fix the origin. In the particular case where the action of the group $G$ on the tangent space at the origin of the abelian variety $A$ is irreducible, we prove that $A$ is isomorphic to the self-product of an elliptic curve and $A/G\simeq \bb P^n$. In the general case, assuming $\dim(A^G)=0$, we prove that $A/G$ is isomorphic to a direct product of projective spaces.\\

\end{abstract}

\maketitle

\section{Introduction}

Quotients of abelian varieties by finite groups have appeared in many dif\-fer\-ent contexts and topics of research. For example, in \cite{KL} Koll\'ar and Larsen study groups acting on simple abelian varieties in dimension greater than or equal to 4, and prove that the quotient has canonical singularities and Kodaira dimension 0. This is done in the context of studying quotients of Calabi-Yau varieties by finite groups. In \cite{IL}, Im and Larsen study the existence of rational curves lying on quotients of abelian varieties by finite groups, and they find a condition on the group that implies that rational curves actually exist on the quotient. 

Along another line, in \cite{Yoshi} Yoshihara initiates the study of Galois em\-bed\-dings of varieties, where he asks when a projective variety embedded into projective space admits a finite linear projection that is a Galois morphism. In particular, the existence of a Galois embedding implies that the variety has a finite group of automorphisms such that the quotient variety is isomorphic to projective space. Yoshihara finishes the paper by analyzing the case of abelian surfaces. In \cite{Auff}, the first author generalizes Yoshihara's results to arbitrary dimension, and proves that if the quotient of an abelian variety by a finite group is projective space, then the abelian variety is isogenous to the self-product of an elliptic curve. As a matter of fact, when there is an action of an irreducible finite subgroup of $\mathrm{GL}(T_0(A))$ with Schur index 1 on an abelian variety $A$, then $A$ is isogenous to the self-product of an elliptic curve, as was proven in \cite{PopovZ}. These results are in some sense opposite to the work done by Koll\'ar and Larsen in \cite{KL}.

These examples show that quotients of abelian varieties by finite groups have indeed garnered attention in varied contexts within algebraic geometry. On the other hand, group actions on abelian varieties over $\C$ lead to the study of finite-dimensional complex representations via their universal covering space, and viceversa. In this sense, a classic article by Looijenga relates root systems and self-products of elliptic curves in \cite{Looijenga}. There is also work by Popov \cite{Popov} and Tokunaga-Yoshida \cite{TY} on complex crystallographic reflection groups, which are extensions $\Gamma$ of a finite complex reflection group $G$ by a $G$-stable lattice $\Lambda$ in $\C^n$. In \cite{TY} the authors study the corresponding quotient $\C^n/\Gamma$ for $n=2$ and in \cite{BS}, Bernstein and Schwarzman do the same in arbitrary dimension for complex crystallographic groups of Coxeter type. Note that such quotients correspond to the quotient of the abelian variety $A=\C^n/\Lambda$ by $G$. However, for a given finite complex reflection group $G$, \emph{not every $G$-stable lattice comes from a complex crystallographic reflection group} and hence the study of smooth quotients of abelian varieties remains an open question.\\

The purpose of this paper is to give a full classification of smooth quotients of abelian varieties by finite groups in the particular case in which the group fixes the origin. Our main theorem states the following:

\begin{theorem}\label{classification}
Let $A$ be an abelian variety of dimension $n\geq 3$, and let $G$ be a (non trivial) finite group of automorphisms of $A$ that fix the origin. Then the following conditions are equivalent:

\begin{itemize}
\item[(1)] $A/G$ is smooth and the action of $G$ on $T_0A$ is irreducible.
\item[(2)] $A/G$ is smooth of Picard number 1.
\item[(3)] $A/G\cong\mathbb{P}^n$.
\item[(4)] There exists an elliptic curve $E$ such that $A\cong E^n$ and $(A,G)$ satisfies exactly one of the following:
\begin{enumerate}[label=(\alph*)]
\item $G\cong C^n\rtimes S_n$ where $C$ is a non-trivial (cyclic) subgroup of automorphisms of $E$ that fix the origin; here the action of $C^n$ is coordinatewise and $S_n$ permutes the coordinates.\label{ex1}
\item $G\cong S_{n+1}$ and acts on 
$$A\cong\{(x_1,\ldots,x_{n+1})\in E^{n+1}:x_1+\cdots+x_{n+1}=0\}$$ 
by permutations.\label{ex2}
\end{enumerate}
\end{itemize}
\end{theorem}

The two cases found in item $(4)$ of the above theorem were studied in detail in \cite{Auff}, where it was proven that both examples give projective space as quotients. This gives the proof of $(4)\Rightarrow(3)$. Our theorem shows that these are the only cases that give smooth quotients in dimension $n\geq 3$. Throughout the paper we will refer to these two examples as Example \ref{ex1} and Example \ref{ex2}, respectively.

Note that the case of dimension $n=1$ is obvious: every pair $(A,G)$ gives $\bb P^1$ as a quotient. For $n=2$, according to Yoshihara (cf.~\cite{Yoshi}), this classification was already done by Tokunaga and Yoshida in \cite{TY}. This paper classifies 2-dimensional complex crystallographic reflection groups. How\-ever, as stated above, these do not cover all possible $G$-stable lattices and hence not all possible group actions on abelian surfaces. The classification in this case was thus incomplete, but was recently achieved by P.~Quezada and the authors in \cite{Pablo}. The outcome is that, in the irreducible case, there is only one example different from Examples \ref{ex1} and \ref{ex2} giving a smooth quotient: it is the pair $(A,G)$ with $A=E^2$ for $E=\C/\Z[i]$ and $G$ is the order 16 subgroup of $\mathrm{GL}_2(\Z[i])$ generated by:
\[\left\{\begin{pmatrix} -1 & 1+i \\ 0 & 1\end{pmatrix}\right.,\, \begin{pmatrix} -i & i-1 \\ 0 & i\end{pmatrix},\, \left.\begin{pmatrix} -1 & 0 \\ i-1 & 1\end{pmatrix} \right\},\]
acting on $A$ in the obvious way.\\

An interesting corollary, which was a first motivation for writing this paper is the following:

\begin{corollary}
If $G$ is a finite group that acts on an abelian variety $A$ such that the elements of $G$ fix the origin and $A/G\cong\mathbb{P}^n$, then $A$ is isomorphic to the self-product of an elliptic curve.
\end{corollary}

The general case is quickly reduced to the irreducible case.

\begin{theorem}[Cf.~Theorem \ref{thm red irred}]\label{thm red irred intro}
Let $G$ be a group that acts by algebraic homomorphisms on an abelian variety $A$ such that $A/G$ is smooth. Assume that $\dim(A^G)=0$. Then $G=\prod_{i=1}^r G_i$, $A=\prod_{i=1}^r A_i$ and each pair $(A_i,G_i)$ satisfies the equivalent conditions from Theorem \ref{classification} above.
\end{theorem}

When $A^G$ has positive dimension, the situation does not necessarily split, but we can still describe the quotient $A/G$ as a fibration over an abelian variety with smooth fibers that are isomorphic to the quotients in Theorem \ref{thm red irred intro}. Actually, we prove in the general case that $A/G$ is smooth if and only if $P_G /G$ is smooth, where $P_G $ is the complementary abelian subvariety of the connected component of $A^G$ that contains 0, cf. Proposition \ref{And smooth}. The notation $P_G $ comes from the fact that in the case that $A$ is the Jacobian of a curve $X$ and $G$ is a group of automorphisms of $X$, $P_G $ is the Prym variety associated to the morphism $X\to X/G$.\\

As an application of our main theorems, we expect to give in a subsequent paper a classification of quotients of principally polarized abelian varieties by groups preserving the divisor class of the polarization. This will be applicable to the specific case of Jacobian varieties with group action coming from an action on the corresponding curve. As a final application, B.~Lim pointed out to us that our classification would be a key ingredient in solving a conjecture by Polishchuk and Van den Bergh (cf.~\cite[Conj,~A]{PVdB}) on semiorthogonal decompositions of categories of equivariant coherent sheaves in the case of abelian varieties.\\

The structure of this paper is as follows: In Section \ref{sec group actions}, we cover some basic properties of abelian varieties with a finite group action and smooth quotient. In particular, we prove in Section \ref{generalities} the implication $(2)\Rightarrow (1)$ from Theorem \ref{classification}, while Section \ref{sect isogeny} is dedicated to the study of $G$-equivariant isogenies in this context, which are used in the sequel. In Section \ref{sec red irred} we prove Theorem \ref{thm red irred intro} and we briefly look at the ultimate general case in which $A^G$ may have positive dimension. Section \ref{reflectiongroup} is dedicated to the proof of $(1)\Rightarrow (4)$ (note that $(3)\Rightarrow (2)$ is evident and $(4)\Rightarrow (3)$ was established in \cite{Auff}, so this concludes the proof of Theorem \ref{classification}). This is the heart of the article and therefore its longest and most technical part. Here we use Shephard-Todd's classification of irreducible complex reflection groups in order to study them case by case. The case of the symmetric group $S_{n}$ is studied in Section \ref{sec Sn} and the infinite family of groups $G(m,p,n)$ for $m\geq 2$ is studied in Section \ref{sec Gmpn}. Finally, Section \ref{sec sporadic} is dedicated to the remaining sporadic cases.\\

\noindent\textit{Acknowledgements:} We would like to thank Anita Rojas and Giancarlo Urz\'ua for interesting discussions and Antonio Behn for help with the computer program SageMath.

\section{Groups acting on abelian varieties with smooth quotient}\label{sec group actions}

\subsection{Generalities}\label{generalities}
Let $A$ be an abelian variety of dimension $n$ and let $G$ be a group of automorphisms of $A$ that fix the origin, such that the quotient variety $A/G$ is smooth. By the Chevalley-Shephard-Todd Theorem, the stabilizer in $G$ of each point in $A$ must be generated by pseudoreflections; that is, elements that fix a divisor pointwise, such that the divisor passes through the point. In particular, $G$ is generated by pseudoreflections and $G$ acts on the tangent space at the origin $T_0(A)$ (this is the analytic representation). In this context, a pseudoreflection is an element that fixes a hyperplane pointwise. We will often abuse notation and display $G$ as either acting on $A$ or $T_0(A)$; it will be clear from the context which action we are considering. 

In what follows, let $\mathcal{L}$ be a fixed $G$-invariant polarization on $A$ (take the pullback of an ample class on $A/G$, for example). For $\sigma$ a pseudoreflection in $G$ of order $r$, define
\begin{align*}
D_\sigma&:=\im(1+\sigma+\cdots+\sigma^{r-1}),\\
E_\sigma&:=\im(1-\sigma).
\end{align*}
These are both abelian subvarieties of $A$.

\begin{lemma}\label{lemma Esigma Dsigma}
We have the following:
\begin{itemize}
\item[1.] $D_\sigma$ is the connected component of $\fix(\sigma):=\ker(1-\sigma)$ that contains 0 and $E_\sigma$ is the complementary abelian subvariety of $D_\sigma$ with respect to $\mathcal{L}$. In particular, $D_\sigma$ is a divisor and $E_\sigma$ is an elliptic curve.
\item[2.] $\sigma$ acts on $E_\sigma$ and hence $r\in\{2,3,4,6\}$.
\item[3.] For $a\not\equiv 0\pmod r$, $E_{\sigma^a}=E_{\sigma}$ and $D_{\sigma^a}=D_{\sigma}$. 
\item[4.] $D_\sigma\cap E_\sigma$ consists of $2$-torsion points for $r=2,4$, of $3$-torsion points for $r=3$ and $D_\sigma\cap E_\sigma=0$ for $r=6$.
\end{itemize}
\end{lemma}

\begin{proof}
Since
\[(1+\sigma+\cdots+\sigma^{r-1})(1-\sigma)=(1-\sigma)(1+\sigma+\cdots+\sigma^{r-1})=1-\sigma^r=0,\]
we see that $D_\sigma\subset\ker(1-\sigma)$ and $E_\sigma\subset\ker(1+\sigma+\cdots+\sigma^{r-1})$. If $x\in\ker(1-\sigma)$, then 
$$rx=x+\sigma(x)+\cdots+\sigma^{r-1}(x)=(1+\sigma+\cdots+\sigma^{r-1})(x)\in D_\sigma,$$
and so after possibly adding an $r$-torsion point to $x$ we obtain that it lies in $D_\sigma$. Therefore both spaces are of the same dimension and, since $D_\sigma$ is irreducible, we get that it corresponds to the connected component containing 0.

To show that $E_\sigma$ is the complementary abelian subvariety of $D_\sigma$, let $H$ be the first Chern class of $\mathcal{L}$, seen as a Hermitian form $H$ on $T_0(A)=\mathbb{C}^n$. Then, since $\sigma$ preserves the numerical class of $\mathcal{L}$, we have that $\sigma^tH=H\sigma^{-1}$. Hence
$$\left(\sum_{i=0}^{r-1}\sigma^i\right)^tH(I_n-\sigma)=H\left(\sum_{i=0}^{r-1}\sigma^{-i}\right)(I_n-\sigma)=0.$$
This shows that the vector subspaces of $T_0(A)$ induced by $D_\sigma$ and $E_\sigma$ are orthogonal with respect to $H$; i.e.~they are complementary abelian subvarieties. This proves 1.

Since $\sigma$ and $(1-\sigma)$ clearly commute, we see that $\sigma(E_\sigma)=E_\sigma$ by definition. This implies immediately that $r\in\{2,3,4,6\}$. This proves 2. For the third assertion, we know that both $D_\sigma$ and $D_{\sigma^a}$ are irreducible divisors. But clearly $\ker(1-\sigma^a)\supset\ker(1-\sigma)$ and hence $D_\sigma=D_{\sigma^a}$. Complementarity implies then that $E_\sigma=E_{\sigma^a}$. Finally, note that since $D_\sigma\subset\ker(1-\sigma)$ and $E_\sigma\subset\ker(1+\sigma+\cdots+\sigma^{n-1})$, for every $x\in D_\sigma\cap E_\sigma$ we have
\[rx=x-\sigma(x)+x+\sigma(x)+\cdots+\sigma^{r-1}(x)=(1-\sigma)(x)+(1+\sigma+\cdots+\sigma^{r-1})(x)=0.\]
This proves that $D_\sigma\cap E_\sigma$ consists of $r$-torsion points. Using the third assertion for $a=2,3$ we prove 4.
\end{proof}

We are now in a position to prove that $(2)\Rightarrow(1)$ in Theorem \ref{classification}; the proof goes along the lines of \cite[Rem.~2.1]{Auff}.

\begin{proposition}
Let $G$ be a finite group acting on an abelian variety $A$ via algebraic homomorphisms. Assume that $A/G$ is smooth and the Picard number of $A/G$ is 1. Then the analytic representation of $G$ is irreducible.
\end{proposition}

\begin{proof}
Assume that $A/G$ is of Picard number 1. We will first show that $G$ does not leave a non-trivial abelian subvariety invariant. Indeed, let $X\subseteq A$ be an abelian subvariety on which $G$ acts, and let $N_X\in\mbox{End}(A)$ be its norm endomorphism with respect to some fixed $G$-invariant polarization $\mathcal{L}$ ($N_X$ on tangent spaces is just the orthogonal projection onto the linear subspace that defines $X$ with respect to the first Chern class of $\mathcal{L}$). Now 
\[N_X^*\mathcal{L}\in\mbox{NS}(A)_\mathbb{Q}^G\cong\mbox{NS}(A/G)_\mathbb{Q}\cong\mathbb{Q},\]
where the subscript $\mathbb{Q}$ indicates that we extended scalars to $\bb Q$. Since $\mathcal{L}\in\mbox{NS}(A)_\mathbb{Q}^G$, we have that $N_X^*\mathcal{L}$ is a rational multiple of $\mathcal{L}$ and therefore the self-intersection number $(N_X^*\mathcal{L})^n$ is non zero. However, by \cite[Prop.~3.1]{ALR}, if $X$ is non-trivial then this number must be zero. Therefore $X$ must be trivial.

Now let $W$ be a $G$-stable linear subspace of $T_0(A)$, and let $\sigma\in G$ be a pseudoreflection. Since the image of $1-\sigma$ is an elliptic curve on $A$ induced, say, by a linear subspace $\langle z_0\rangle\leq T_0(A)$, we have that for every $z\in W$, $(1-\sigma)(z)=\lambda_z z_0$ for some $\lambda_z\in\mathbb{C}$.

If $\lambda_z\neq 0$ for some $z\in W$, then $z_0\in W$. Now, since the translates of $z_0$ by $G$ all lie in $W$ and $\sum_{\tau\in G}\tau(E_\sigma)=A$ by the previous discussion, we have that $W=T_0(A)$. 

Assume now that $\lambda_z=0$ for every $z\in W$ and every pseudoreflection $\sigma\in G$. In particular, $W$ is fixed by every $\sigma$, and since these generate the group, we have that $W$ is fixed pointwise by $G$. Now since $G$ does not fix pointwise any non-trivial abelian subvariety of $A$, we have that
$$\bigcap_{\tau\in G}\ker(1-\tau)\subseteq A$$
is finite and so its preimage in $T_0(A)$ is discrete. However $W$ is contained in this preimage, and so it must be trivial. 
\end{proof}

\subsection{$G$-equivariant isogenies}\label{sect isogeny}
We will consider now a new abelian variety $B$ equipped with a $G$-equivariant isogeny to $A$, which we will call a $G$-isogeny from now on. Let $\Lambda_A$ denote the lattice in $\C^n$ such that $A=\C^n/\Lambda_A$. Let $\Lambda_B\subseteq\Lambda_A$ be a $G$-invariant sublattice, and let $B:=\C^n/\Lambda_B$ be the induced abelian variety, along with the $G$-isogeny
\[\pi:B\to A,\]
whose analytic representation is the identity. Note that this implies that $\sigma\in G$ is a pseudoreflection of $B$ if and only if it is a pseudoreflection of $A$. We may then consider the subvarieties $E_\sigma,D_\sigma\subset A$ defined as above, which we will denote by $E_{\sigma,A} $ and $D_{\sigma,A}$. Now, we can do the same thing for $B$ and hence we obtain subvarieties $E_{\sigma,B},D_{\sigma,B}\subset B$. Note that, by definition, $\pi$ sends $E_{\sigma,B}$ to $E_{\sigma,A}$ and $D_{\sigma,B}$ to $D_{\sigma,A}$.\\

Define $\Delta:=\ker(\pi)$. Since $\pi$ is $G$-equivariant, $G$ acts on $\Delta$ and hence we may consider the group $\Delta\rtimes G$. This group acts on $B$ in the obvious way: $\Delta$ acts by translations and $G$ by automorphisms. In particular, we see that the quotient $B/(\Delta\rtimes G)$ is isomorphic to $A/G$.

Our goal is to reduce as much as we can the structure of $B/G$ and $\Delta$ and to prove that the latter must be trivial in several cases. Fix then a pseudoreflection $\sigma\in G$ of order $r$ and consider the subvarieties $E_{\sigma,A},D_{\sigma,A}\subset A$ and $E_{\sigma,B},D_{\sigma,B}\subset B$. Define moreover $F_{\sigma,A}=E_{\sigma,A}\cap D_{\sigma,A}$ and $F_{\sigma,B}$ similarly. Then the isogeny $\pi:B\to A$ sends $F_{\sigma,B}$ to $F_{\sigma,A}$.

\begin{lemma}
Assume that the map $E_{\sigma,B}\to E_{\sigma,A}$ is injective and that the map $F_{\sigma,B}\to F_{\sigma,A}$ is surjective. Then $\Delta\subset D_{\sigma,B}$.
\end{lemma}

\begin{proof}
Since $E_{\sigma,B}$ and $D_{\sigma,B}$ generate $B$, we have $z=x+y\in\Delta=\ker(\pi)$ with $x\in E_{\sigma,B}$ and $y\in D_{\sigma,B}$ for every $z\in\Delta$. Then $\pi(z)=0$ implies $\pi(x)=-\pi(y)\in F_{\sigma,A}$. But since $F_{\sigma,B}\to F_{\sigma,A}$ is surjective and $E_{\sigma,A}\to E_{\sigma,B}$ is injective, we have that $x\in E_{\sigma,B}\cap\pi^{-1}(F_{\sigma,A})=F_{\sigma,B}$. Thus $x\in D_{\sigma, B}$ and hence $\Delta\subset D_{\sigma,B}$.
\end{proof}

Since all conjugates of a pseudoreflection are pseudoreflections and everything is $G$-equivariant, we immediately get the following result.

\begin{proposition}\label{prop Delta and D}
Let $\sigma\in G$ be a pseudoreflection and assume that the map $E_{\sigma,B}\to E_{\sigma,A}$ is injective and that the map $F_{\sigma,B}\to F_{\sigma,A}$ is surjective. Then the subgroup $\Delta=\ker(\pi)$ is contained in $D_{\tau\sigma\tau^{-1},B}$ for \emph{every} $\tau\in G$.\qed
\end{proposition}

We conclude this section by studying pseudoreflections in $\Delta\rtimes G$.

\begin{lemma}\label{main lemma}
Let $\sigma\in\Delta\rtimes G$ be a pseudoreflection. Then $\sigma=(t,\tau)$ with $\tau\in G$ a pseudoreflection and $t\in\Delta\cap E_{\tau,B}$.
\end{lemma}

\begin{proof}
Let $t\in\Delta$ and $\tau\in G$ be such that $\sigma=(t,\tau)\in\Delta\rtimes G$. This element acts on $B$ sending $x$ to $\tau(x)+t$. By definition, $\sigma$ must fix a divisor pointwise, that is, there is a subvariety $C\subset B$ of codimension 1 such that $x=\tau(x)+t$ for all $x\in C$, or equivalently, $x\in (1-\tau)^{-1}(t)$. But since $1-\tau\in\End(B)$, we see that $C$ is a translate of $\ker(1-\tau)$, which is a divisor if and only if $\tau$ is a pseudoreflection and $t\in (1-\tau)(B)=E_{\tau,B}$.
\end{proof}

\subsection{Reduction to irreducible representations}\label{sec red irred}
Let $G$ be a group that acts by algebraic homomorphisms on an abelian variety $A$ such that $A/G$ is smooth. In particular the analytic representation of $G$ on $T_0(A)$ is a finite complex reflection group. It is well-known (cf.~for instance \cite{ST} or \cite[\S1.4]{Popov}), that $G\cong G_1\times\cdots\times G_r$ and $T_0(A)=W_0\oplus W_1\oplus\cdots\oplus W_r$ where 
\begin{itemize}
\item $W_i$ is an irreducible complex representation of $G_i$ that makes $G_i$ an irreducible finite complex reflection group for $i>0$;
\item $G_j$ acts trivially on $W_i$ for $i\neq j$.
\end{itemize}
In particular, $W_0=T_0(A)^G$.

\begin{lemma}\label{lemma desc Ai}
The subspace $W_i$ induces a $G$-stable abelian subvariety $A_i$ of $A$ such that $G_j$ acts trivially on $A_i$ for $i\neq j$. Moreover, $A_i/G=A_i/G_i$ is smooth.
\end{lemma}

\begin{proof}
Since $W_0=T_0(A)^G$, then $A_0$ is the neutral connected component of $A^G$ and $A_0/G=A_0$. Assume now $i>0$, let $\sigma\in G_i$ be a pseudoreflection and let $L$ be the linear subspace of $T_0(A)$ that induces $E_\sigma$. It is clear that $L\subseteq W_i$, since $L=(1-\sigma)(T_0(A))$. Since the representation of $G_i$ on $W_i$ is irreducible, we have that 
$$W_i=\sum_{\tau\in G}(\tau(L)).$$
Therefore, $W_i$ is the tangent space of the abelian subvariety $A_i=\sum_{\tau\in G}\tau(E_\sigma)$. It is clear that $A_i$ is $G$-stable and $G_j$ acts trivially on $A_i$ for $i\neq j$ so that $A_i/G=A_i/G_i$. Moreover, since $\stab_{G_i}(x)=\stab_G(x)\cap G_i$ for $x\in A_i$ and every pseudoreflection in $G$ belongs to some $G_j$, it is easy to see that $\stab_{G_i}(x)$ is generated by pseudoreflections in $G_i$ whenever $\stab_G(x)$ is generated by pseudoreflections in $G$. This is the case by the Chevalley-Shephard-Todd Theorem because $A/G$ is smooth and therefore $A_i/G_i$ is smooth.
\end{proof}

We can now prove that, whenever $A_0$ is trivial, it is enough to understand the case when the action of $G$ on $T_0(A)$ is irreducible.

\begin{theorem}\label{thm red irred}
Let $G$ be a group that acts by algebraic homomorphisms on an abelian variety $A$ such that $A/G$ is smooth. Assume that $\dim(A^G)=0$. Then $A$ is the direct product of the $A_i$, defined as above. In particular, 
\[A/G\cong A_1/G_1\times\cdots\times A_r/G_r.\]
\end{theorem}

We will need the following small result on irreducible finite complex reflection groups:

\begin{lemma}\label{lem surj}
Let $G$ be a finite complex reflection group acting irreducibly on $\C^n$. Then there exists $\tau\in G$ such that $(1-\tau)$ is surjective.
\end{lemma}

\begin{proof}
This amounts to finding an element $\tau\in G$ such that 1 is not an eigenvalue of $\tau$. Now this follows directly from \cite[Thm.~5.4]{ST}.
\end{proof}

\begin{proof}[Proof of Theorem \ref{thm red irred}]
Consider the subvarieties $A_i\subset A$ from Lemma \ref{lemma desc Ai} for $i\geq 1$ ($A_0$ is trivial by the hypothesis on $A^G$). Then there is a natural $G$-isogeny
\[B:=A_1\times\cdots\times A_r\to A,\]
given by the sum in $A$. In particular, the kernel of this isogeny is
\[\Delta:=\left\{(a_1,\ldots,a_r)\in A_1\times\cdots\times A_r\mid \sum_{i=0}^r a_i=0\right\}.\]
We claim that $\Delta$ is fixed pointwise by $G$. Indeed, since $a_i\in A_i$, we know that $G_j$ acts trivially on it for $j\neq i$; but since $a_i=-\sum_{j\neq i}a_j\in \sum_{j\neq i}A_j$, we also know that $G_i$ acts trivially on it (since it acts trivially on every $A_j$ for $j\neq i$). We see then that $G$ acts trivially on every coordinate of every element of $\Delta$, which proves the claim.

Thus, $\Delta\times G$ acts on $B$ and hence $A/G$ is isomorphic to $B/(\Delta\times G)$, i.e.
\[A/G\cong [(A_1/G_1)\times\cdots (A_r/G_r)]/\Delta.\]
All we need to prove now is that $\Delta$ has to be trivial. Assume then that this is not the case and note that the action of $(a_1,\ldots,a_r)\in\Delta$ on $X:=(A_1/G_1)\times\cdots (A_r/G_r)$ corresponds coordinatewise to the action of $a_i$ on $A_i/G_i$ (which is well defined since $a_i$ is $G_i$-invariant and thus its action commutes with that of $G_i$). Now, the action of $a_i$ on $A_i/G_i$ always has a fixed point $p_i$. Indeed, by Lemma \ref{lem surj} we know that there exists $\tau\in G_i$ such that $(1-\tau)$ is surjective. Thus, there exists $x_i\in A_i$ such that $x_i-\tau(x_i)=a_i$, which implies that the image $p_i$ of $x_i$ in $A_i/G_i$ is fixed by $a_i$. We see then that $(p_1,\ldots,p_r)\in X$ is a point that is fixed by $(a_1,\ldots,a_r)$ and thus the action of $\Delta$ on $X$ is not free. It is also a non-trivial action since the image of $0\in B$ in $X$ is clearly moved by $\Delta$.

Since $A/G=X/\Delta$ is smooth, the Chevalley-Shephard-Todd Theorem tells us then that every stabilizer of this action has to be generated by pseudoreflections. Now this is impossible since, for every non-trivial $(a_1,\ldots,a_r)\in\Delta$, its fixed locus in $X$ corresponds to the product of the fixed loci in each $A_i/G_i$ via $a_i$. We see then that if any element in $\Delta$ is a pseudoreflection, it must fix all but one $A_i/G_i$ (otherwise the fixed locus would not be a divisor), which amounts to $a_i=0$ for all but one $i$, and this is impossible since $\sum_{i=1}^ra_i=0$. This proves that $\Delta$ is trivial.
\end{proof}

Let us consider now the ``degenerated'' case in which $\dim(A^G)>0$.

\begin{proposition}\label{And smooth}
Let $G$ be a group that acts by algebraic homomorphisms on an abelian variety $A$. Let $A_0$ be the connected component of $A^G$ containing 0 and let $P_G $ be its complementary abelian subvariety with respect to a $G$-invariant polarization. Then there exists a fibration $A/G\to A_0/(A_0\cap P_G )$ with fibers isomorphic to $P_G /G$. Moreover, $A/G$ is smooth if and only if $P_G /G$ is smooth.
\end{proposition}

\begin{proof}
Consider as in the last proof the natural $G$-isogeny $A_0\times P_G \to A$ and denote its kernel by $\Delta$. This can be rewritten as
\[A\cong A_0\stackrel{\Delta}{\times}P_G .\]
Now, the same argument from the proof above shows that $\Delta$ is fixed pointwise by $G$. In particular, the actions of $G$ and $\Delta$ on $P_G $ commute and it is easy to see then that
\[A/G\cong A_0\stackrel{\Delta}{\times} (P_G /G).\]
Recalling that $\Delta\cong A_0\cap P_G $, we may thus see $A/G$ as a fibration over the abelian variety $A_0/(A_0\cap P_G )$ with fibers isomorphic to $P_G /G$.

Finally note that, since the action of $\Delta$ on $A_0$ is free, the quotient $A/G=(A_0\times P_G /G)/\Delta$ is smooth whenever $P_G /G$ is. On the other hand, by the same argument we used for $A_i/G_i$, $P_G /G$ is smooth if $A/G$ is.
\end{proof}

Note that this fibration is non-trivial in general, as shown by the following example: Let $E$ be an elliptic curve and let $e\in E[2]$. Define $B=E\times E$ and let $G=\{\pm 1\}$ act on the second factor. Note in particular that $(e,e)$ is $G$-invariant. Put then $A=B/\langle (e,e)\rangle$ and denote by $\pi:B\to A$ the projection. We have that $A_0=\pi(E\times\{0\})$, $P_G =\pi(\{0\}\times E)$, $B=A_0\times P_G $ and $\Delta=\langle (e,e)\rangle$. We see then that
\[A/G\cong B/(\Delta\times G)\cong (B/G)/\Delta\cong (E\times\bb P^1)/\Delta,\]
where, up to a base change in $\bb P^1$, $\Delta$ acts on $E\times\bb P^1$ by sending $(x,y)$ to $(x+e,-y)$. Looking at the first coordinate, we see then that the action is free and thus defines by \'etale descent a non-trivial $\bb P^1$-bundle over the elliptic curve $E/\langle e\rangle$.

\section{Quotients by irreducible finite complex reflection groups}\label{reflectiongroup}

Given the results from the last section, we will now concentrate on group actions on abelian varieties that satisfy the following condition (which is condition (1) from Theorem \ref{classification}):
\begin{equation}\label{eq cond star}
A/G\text{ is smooth and the analytic representation of }G\text{ is irreducible.}\tag{$\star$}
\end{equation}
If the pair $(A,G)$ satisfies \eqref{eq cond star}, we see that the analytic representation makes $G$ an irreducible finite complex reflection group, in the sense of Shephard-Todd \cite{ST}. These groups were completely classified by Shephard and Todd in \cite{ST}, where they discovered that any finite irreducible complex reflection group is either a group $G(m,p,n)$ depending on $m,p,n\in\mathbb{Z}_{>0}$ where $p\mid m$ and $n\geq 1$, or is one of 34 sporadic cases. The group $G(m,p,n)$ consists of the semidirect product $H\rtimes S_n$ of the abelian group 
\begin{equation}\label{eq def H}
H=H(m,p,n)=\{(\zeta_m^{a_1},\ldots,\zeta_m^{a_n})\mid a_1+\cdots+a_n\equiv0\pmod p\}\subset \mu_m^n
\end{equation}
with the symmetric group $S_n$, where $\zeta_m$ is a primitive $m$-th root of unity and $S_n$ acts on each member by permuting the coordinates in the obvious way. For $m=p=1$, $G(1,1,n)$ is just the symmetric group on $n$ letters, and acts irreducibly on an $(n-1)$-dimensional complex vector space. For $m>1$, $G(m,p,n)$ acts irreducibly on an $n$-dimensional complex vector space.

The purpose of this section is to describe which of these actions actually appear on abelian varieties of dimension $n\geq 3$ such that \eqref{eq cond star} is satisfied. In the following subsections we will analyze each case of the Shephard-Todd classification. In particular, in this section we prove $(1)\Rightarrow(4)$ of Theorem \ref{classification}.

\subsection{The case $m=p=1$: the standard representation of $S_{n+1}$}\label{sec Sn}
Let $G(1,1,n+1)=S_{n+1}$ act on an abelian variety $A$ of dimension $n\geq 2$ in such a way that its action on $T_0(A)$ is the standard one. Let $\sigma=(1\, 2)$ and $E=E_{\sigma}$ be induced by a line $L_\sigma\subseteq T_0(A)$, and define the lattice
\[\Lambda_B:=\sum_{\tau\in S_{n+1}}\tau(L_\sigma\cap\Lambda_A).\]
This gives us a $G$-invariant sublattice of $\Lambda_A$, and we therefore get a $G$-equivariant isogeny $\pi:B\to A$ with kernel $\Delta$. Applying this construction to Example \ref{ex2}, we see that it gives the whole lattice and hence corresponds to Example \ref{ex2} itself. We can thus see $B$ as 
\[B=\{(x_1,\ldots,x_{n+1})\in E^{n+1}\mid x_1+\cdots+x_{n+1}=0\}\]
and $S_{n+1}$ acts coordinatewise in the natural way. Using the notations from Section \ref{sect isogeny}, we see by inspection that $F_{\sigma,B}=E_{\sigma,B}[2]\cong E[2]$, hence the map $\pi:F_{\sigma,B}\to F_{\sigma,A}$ is surjective since by Lemma \ref{lemma Esigma Dsigma} we have $F_{\sigma,A}\subset E_{\sigma,A}[2]\cong E[2]$. Moreover, the induced map $E_{\sigma,B}\to E_{\sigma,A}$ is injective by construction. Thus, by Proposition \ref{prop Delta and D}, we have that $\Delta$ is contained in the fixed locus of all the conjugates of $\sigma$. In other words, $\Delta$ consists of elements of the form $(x,\ldots,x)\in E^{n+1}$ such that $(n+1)x=0$. Note that this implies that the direct product $\Delta\times G$ acts on $B$.

\begin{proposition}\label{prop Sn}
Let $n\geq 2$. If $S_{n+1}$ acts on $A$ in such a way that its analytic representation is the standard representation and $(A,S_{n+1})$ satisfies \eqref{eq cond star}, then $A\cong E^n$ and $S_{n+1}$ acts as in Example \ref{ex2}.
\end{proposition}

\begin{proof}
Let $\pi:B\to A$ be the $G$-isogeny defined above. We have to prove then that $\Delta=\{0\}$.
Let $\bar t=(t,\ldots,t)\in\Delta$ be a non-trivial element and let $\tau\in G$ be an element such that $(1-\tau)$ is surjective (such an element exists by Lemma \ref{lem surj}). Then there exists an element $z\in B$ such that $z-\tau(z)=\bar t$ and thus the stabilizer of $z$ contains the element $(\bar t,\tau)\in\Delta\times G$.

Note now that $\Delta\cap E_{\sigma,B}=\{0\}$ for every pseudoreflection $\sigma\in G$. Thus, by Lemma \ref{main lemma}, the only pseudoreflections in $\Delta\times G$ are the transpositions in $G=S_{n+1}$, and so $\stab_G(z)$ cannot be generated by pseudoreflections. Therefore if $\Delta\neq 0$, $A/G$ is not smooth by the Chevalley-Shephard-Todd Theorem, which contradicts condition \eqref{eq cond star}.
\end{proof}

\subsection{The case of $G(m,p,n)$, $m\geq 2$, $n\geq 3$}\label{sec Gmpn}
Now we will study when $G=G(m,p,n)$ acts on an abelian variety $A$ of dimension $n$ for $m\geq2$. \emph{We assume here that $n\geq 3$} (recall that the case of dimension 2 was already dealt with elsewhere). Recall that $G=H\rtimes S_n$, where $H\subset \mu_m^n$ is defined in \eqref{eq def H} and it acts coordinatewise on $\C^n=T_0(A)$, while $S_n$ permutes the variables in the obvious way.

\begin{rem}
In what follows, we will try as much as we can to prove results on $G$ without splitting into subcases depending on the value of $p$. Hence, in the following arguments we will only consider elements in $G(m,m,n)\subset G(m,p,n)$, even if in some cases a simpler argument can be found for certain values of $p$. We will also keep all arguments (with one exception) depending on at most three dimensions, so that they are all valid for $n\geq 3$.
\end{rem}

Let $E_i$ be the image of $\mathbb{C}e_i$ in $A$ via the exponential map. We claim that it corresponds to an elliptic curve. Indeed, consider the element $\tau=(1,\zeta_m,\zeta_m^{-1},1,\ldots,1)\in H$ and denote $\rho=1+\tau+\cdots+\tau^{m-1}$. Then a direct computation shows that, for $\sigma=(1\, 2)\in S_n\subset G$, $\im(\rho(1-\sigma))=\C e_1$. This tells us that $E_1=\rho(1-\sigma)(A)$ and hence it corresponds to an elliptic curve. This allows us to prove the following.

\begin{lemma}\label{lem m leq 6}
Assume that $G$ acts on $A$ as above. Then $m\in\{2,3,4,6\}$ and, if $m\geq 3$, then the curves $E_i\subset A$ have non-trivial automorphisms.
\end{lemma}

\begin{proof}
Consider the curve $E_1\subset A$ defined as above. We see then that the element $(\zeta_m,\zeta_m^{-1},1,\ldots,1)\in H$ induces an automorphism of order $m$ of $E_1$. Therefore $m\in\{2,3,4,6\}$ and, if $m\geq 3$, then $E_1$ has non-trivial automorphisms. The other $E_i$ are obtained from $E_1$ via the action of $S_n$ and hence are isomorphic to it.
\end{proof}

Now, let $\Lambda_A$ be a lattice for $A$ in $\mathbb{C}^n$. Then $\mathbb{C}e_i\cap\Lambda_A$ corresponds to the lattice of $E_i$ in $\C=\C e_i$. We can thus define the $G$-stable sublattice of $\Lambda_A$
$$\Lambda_B:=\bigoplus_{i=1}^n(\mathbb{C}e_i\cap\Lambda_A).$$
As in Section \ref{sect isogeny}, this defines a $G$-isogeny $\pi:B\to A$. Moreover, we see that $B\cong E_1\times\cdots\times E_n\cong E^n$ and that $\pi|_{E_i}$ is an injection. As in the previous section, let $\Delta$ be the kernel of $\pi$. We will study the different possible quotients $A/G$ by studying the possible quotients $B/(\Delta\rtimes G)$ and thus by studying the possible $\Delta$'s. Let us start with the case of a trivial $\Delta$:

\begin{proposition}\label{prop Gmpn p not 1}
Let $G=G(m,p,n)$ with $n\geq 2$ act on $B=E^n$ as above. Then the quotient $B/G$ is smooth if and only if $p=1$.
\end{proposition}

\begin{proof}
By Lemma \ref{lem m leq 6}, we know that $m\in\{2,3,4,6\}$. Thus, if $p=1$, the action of $G$ on $B$ is by construction the same as in Example \ref{ex1}, which tells us that $B/G\cong\bb P^n$ and hence it is smooth.\\

Assume now that $p\geq 2$. By Lemma \ref{lem m leq 6}, we also know that if $m\neq 2$ then $E$ has non-trivial automorphisms given by multiplication by $\zeta_m$. In particular, $E$ is a very specific curve in each of these cases and it is easy to see that:
\begin{itemize}
\item if $m=3,6$, then there exists a non-trivial $t\in E[3]$ such that $\zeta_6 t=-t$;
\item if $m=4$, then there exists a non-trivial $t\in E[2]$ such that $\zeta_4 t=t$.
\end{itemize}
Consider one such element $t\in E$ unless $(m,p)\in\{(2,2),(6,2)\}$, in which case take any non-trivial element $t\in E[2]$. Let $(x_3,\ldots,x_{n})\in E^{n-2}$ be a general element. Then, if $\bar x=(t,0,x_3,\ldots,x_{n})\in B=E^n$, we immediately see that an element in $\stab_G(\bar x)$ must be in $H\subset G$ since the coordinates cannot be permuted, even after applying automorphisms on some coordinates via $H$. A direct computation tells us then that $\stab_G(\bar x)$ is equal to the (abelian) subgroup of $H\subset G$ given in each case by the following table:\\

\begin{center}
\begin{tabular}{ c | l  }
$(m,p)$ & Generators of $\mbox{Stab}_G(\bar x)$\\\hline
(2,2) & $(-1,-1,1,\ldots,1)$\\
(3,3) & $(\zeta_3,\zeta_3^{-1},1,\ldots,1)$\\
(4,2) & $(\zeta_4,\zeta_4,1,\ldots,1)$, $(-1,1,1,\ldots,1)$, $(1,-1,1,\ldots,1)$\\
(4,4) & $(\zeta_4,\zeta_4^{-1},1,\ldots,1)$\\
(6,2) & $(-1,-1,1,\ldots,1)$, $(1,\zeta_3,1,\ldots,1)$\\
(6,3) & $(\zeta_3,\zeta_3^{-1},1,\ldots,1)$, $(1,-1,1,\ldots,1)$\\
(6,6) & $(\zeta_3,\zeta_3^{-1},1,\ldots,1)$
\end{tabular}
\end{center}

\vspace{1em}

However we observe that in all cases the first element is not a pseudoreflection, since its fixed locus is of codimension 2. Moreover, the only pseudoreflections in $\stab_G(\bar x)$ are the other given generators (and their powers) and hence they cannot generate the first one. Therefore, by the Chevalley-Shephard-Todd Theorem, the quotient $B/G$ is not smooth.
\end{proof}

Let us consider now the case of a non-trivial kernel $\Delta$. We start with an application of Proposition \ref{prop Delta and D}.

\begin{lemma}\label{lemma Delta diagonal or hyper}
If $\Delta$ is non-trivial, then $m\neq 6$ and, if we define the following type of elements in $\Delta$:
\begin{itemize}
\item Diagonal: $(t,\ldots,t)$ with $t\in E$;
\item Hyperplanar: $(t,-t,0,\ldots,0)$ with $t\in E$;
\end{itemize}
then $\Delta$ contains a non-trivial hyperplanar element \emph{unless} it consists purely of diagonal elements. Moreover, the coordinates of every hyperplanar element are invariant by $\zeta_m$, so in particular these elements are 2-torsion if $m=2,4$ and 3-torsion if $m=3$.
\end{lemma}

\begin{proof}
Let $\sigma=(1\,2)$ and note that $(t,-t,0,\ldots,0)\in E_{\sigma,B}$. Then there being no non-trivial hyperplanar element in $\Delta$ amounts to $E_{\sigma,B}\to E_{\sigma,A}$ being an isomorphism. By inspection, we see that $F_{\sigma,B}=E_{\sigma,B}[2]$ and we can thus apply Proposition \ref{prop Delta and D}, which tells us that elements in $\Delta$ are invariant by \emph{every} transposition, hence diagonal.

Assume now that $\Delta$ contains a hyperplanar element $\bar t$. Then, since $\Delta$ is $G$-stable, we have that, for $\rho_1=(\zeta_m,1,\zeta_m^{-1},1,\ldots,1)\in H$,
\[(1-\rho_1)(\bar t)=((1-\zeta_m)t,0,\ldots,0)\in\Delta.\]
But, by construction, there are no elements of the form $(x,0,\ldots,0)$ in $\Delta$. We deduce then that $t$ is $\zeta_m$-invariant. The assertion on the torsion of its coordinates follows immediately.

Assume finally that $m=6$ and let $(t_1,\ldots,t_n)\in\Delta$. Define $\sigma_i=(1\, i)\in S_n\subset G$ and $\rho_2=(\zeta_6^{-1},\zeta_6,1,\ldots,1)\in H\subset G$. Then
\[[(1-\rho_2)(1-\rho_1)\sigma_i](\bar t)=(t_i,0,\ldots,0)\in\Delta,\]
which implies as above that $t_i=0$ and thus $\Delta=0$.
\end{proof}

Let us study now pseudoreflections in $\Delta\rtimes G$. Define the elements
\begin{align*}
\rho &:=(\zeta_m,\zeta_m^{-1},1,\ldots,1)\in H\subset G;\\
\sigma &:=(1\, 2)\in S_n\subset G;\\
\tau &:=(\zeta_m^p,1,\ldots,1)\in H\subset G.
\end{align*}
Then there are two types of pseudoreflections in $G$:\label{pseudoref in Delta G}
\begin{itemize}
\item[(I)] conjugates of $\rho^a\sigma$ for $0\leq a<\frac mp$;
\item[(II)] conjugates of powers of $\tau$ (these do not exist if $m=p$);
\end{itemize}
and the corresponding elliptic curves in $B$ are respectively:
\begin{align*}
E_{\rho^a\sigma}&=\{(x,-\zeta_m^ax,0,\ldots,0)\mid x\in E\};\\
E_{\tau}&=\{(x,0,0,\ldots,0)\mid x\in E\}.
\end{align*}
Note now that elements of the form $(x,0,\ldots,0)$ are not in $\Delta$ by construction of the isogeny $\pi:B\to A$. Using Lemmas \ref{main lemma} and \ref{lemma Delta diagonal or hyper}, we see then that pseudoreflections in $\Delta\rtimes G$ that are not in $G$ must be of the form
\begin{itemize}
\item[(III)] conjugates of $(\bar t,\rho^a\sigma)\in\Delta\rtimes G$ for $0\leq a<p$;
\end{itemize}
where $\bar t=(t,-t,0,\ldots,0)\in\Delta$ and $t$ is $\zeta_m$-invariant.\\

With these considerations, we can restrict further the structure of $\Delta$. For instance, diagonal elements in $\Delta$ are bound to bring problems since they do not belong to any elliptic curve $E_\upsilon$ for a pseudoreflection $\upsilon\in G$. Thus, they cannot bring up new pseudoreflections in $\Delta\rtimes G$ unless they are generated by hyperplanar elements. This is explained by the following proposition.

\begin{proposition}[$\Delta$ is not diagonal]\label{proposition Delta not diagonal}
Assume that there exists $s\in E$ such that $(s,\ldots,s)\in \Delta$ but $(s,-s,0,\ldots,0)\not\in\Delta$. Then $A/G$ is not smooth.
\end{proposition}

In particular, we see that $\Delta$ has to contain at least one hyperplanar element.

\begin{proof}
Since $A/G\cong B/(\Delta\rtimes G)$, we will work with this last quotient using the Chevalley-Shephard-Todd Theorem.

Let $\bar s\in\Delta$ denote the diagonal element in the statement of the Proposition. We will prove first that an element of the form $(\bar s,\upsilon)$ cannot be generated by pseudoreflections in $\Delta\rtimes G$. Indeed, the only pseudoreflections that are not in $G$ are those of type (III), so that if $(\bar s,\upsilon)$ was generated by pseudoreflections, we should be able to write
\begin{equation}\label{eq diagonal delta}
\bar s=\sum_{i=1}^\ell \upsilon_i(\bar t_i),
\end{equation}
with $\bar t_i=(t_i,-t_i,0\ldots,0)\in\Delta$ a hyperplanar element and $\upsilon_i\in G$. In particular, $\bar s$ would be contained in the sub-$G$-module of $\Delta$ generated by the $\bar t_i$. But since $t_i$ is $\zeta_m$-invariant, the only way in which $G$ acts on the $\bar t_i$ is by permuting their coordinates. Thus, by looking at the first coordinate in equation \eqref{eq diagonal delta}, we get that $s$ is a linear combination of the $t_i$, which implies immediately that $(s,-s,0,\ldots,0)$ is a linear combination of the $\bar t_i$ and hence is in $\Delta$, contradicting our hypothesis.

Having proved this, it suffices then to exhibit an element $\bar x\in B$ such that its stabilizer in $\Delta\rtimes G$ has an element of the form $(\bar s,\upsilon)$. In other words, we need $\upsilon\in G$ and $\bar x\in B$ such that $\upsilon(\bar x)+\bar s=\bar x$, and this is a direct consequence of Lemma \ref{lem surj}.
\end{proof}

Denote by $E_0$ the subgroup of $\zeta_m$-invariant elements of $E$. Then $E_0$ is equal to $E[2]\cong (\Z/2\Z)^2$ if $m=2$, isomorphic to $\Z/3\Z$ if $m=3$ and isomorphic to $\Z/2\Z$ if $m=4$. Now that we know that diagonal elements in $\Delta$ only appear if generated by hyperplanar elements, Lemma \ref{lemma Delta diagonal or hyper} tells us that $\Delta$ is contained in $E_0^n=B^H$, and more precisely in the ``hyperplane''
\[\left\{(x_1,\ldots,x_n)\in E_0^n\mid \sum_{i=1}^nx_i=0\right\}\subset E_0^n\subset B.\]
Indeed, for $m=3,4$ the mere presence of a hyperplanar element implies by $G$-stability that $\Delta$ is actually the whole ``hyperplane'' and thus the presence of any additional element in $\Delta$ would imply the existence of elements of the form $(x,0,\ldots,0)$, which is forbidden by construction. A similar argument using Proposition \ref{proposition Delta not diagonal} works for $m=2$. In this last case, one hyperplanar element does not suffice to generate the whole hyperplanar subgroup of $E_0^n$ since $E_0=E[2]$ needs two generators. We prove now that such an ``incomplete'' hyperplanar $\Delta$ does not work either.

\begin{proposition}\label{proposition Delta not special}
Assume that $m=2$, $\Delta\neq\{0\}$ and there exists a hyperplanar 2-torsion element that is not in $\Delta$. Then $A/G$ is not smooth.
\end{proposition}

\begin{proof}
As before, we can use the Chevalley-Shephard-Todd Theorem on the quotient $B/(\Delta\rtimes G)\cong A/G$.

By the previous Proposition, we may assume that $\Delta$ has a non-trivial element $\bar t=(t,t,0,\ldots,0)\in\Delta$ with $t\in E[2]$. But since $E[2]\cong (\Z/2\Z)^2$, we easily see from the hypothesis that there are no elements of the form $(s,s,0,\ldots,0)$ for $s\neq 0,t$.

Let $s\in E[2]$ be such an element. Let $t_1\in E[4]$ be such that $2t_1=t$ and let $t_2=t_1+s\in E[4]$. Let $(x_3,\ldots,x_n)\in E^{n-2}$ be a general element and consider the element $\bar x=(t_1,t_2,x_3,\ldots,x_n)\in B$. Recalling the notations given in page \pageref{pseudoref in Delta G}, it is easy to see that $(\bar t,\rho)\in\Delta\rtimes G$ fixes $\bar x$. Since $t_1\neq\pm t_2$, it is also easy to see that no element in $G$ fixes $\bar x$, so that pseudoreflections fixing $\bar x$ can only be of type (III), that is either $(\bar t,\sigma)$ or $(\bar t,\rho\sigma)$. But again, since $t_1\neq\pm t_2$, we see that neither of these fixes $\bar x$. Thus, $\stab_{\Delta\rtimes G}(\bar x)$ is not generated by pseudoreflections and hence $B/(\Delta\rtimes G)$ cannot be smooth by the Chevalley-Shephard-Todd Theorem.
\end{proof}

Thus, we are reduced to the ``full'' hyperplanar case, that is,
\begin{equation}\label{eq hyper delta}
\Delta:=\left\{(x_1,\ldots,x_n)\in E_0^n\mid \sum_{i=1}^nx_i=0\right\}\subset E_0^n\subset B.
\end{equation}
We prove then the following:

\begin{proposition}[$\Delta$ is not hyperplanar]\label{proposition Delta not hyperplanar}
Assume that $\Delta$ is as in \eqref{eq hyper delta}. Then $A/G$ is not smooth \textbf{except} if $G=G(2,2,3)$.
\end{proposition}

\begin{proof}
As always, it will suffice to give an element $\bar x\in B=E^n$ such that its stabilizer in $\Delta\rtimes G$ is not generated by pseudoreflections. The idea, as in the last proof, is to exhibit an element whose coordinates $x_i$ are ``different enough'' so that it is clear that elements in $S_n\subset G$ cannot appear in $\stab_{\Delta\rtimes G}(\bar x)$, even after being mixed up with elements of $\Delta\times H\subset \Delta\rtimes G$. This amounts to ensuring that different coordinates do not belong to the same $(E_0\times\mu_m)$-orbit (this is how $\Delta\times H$ acts on coordinates). Then the stabilizer must be contained in $\Delta\times H$ and hence it is easy to exhibit examples that are not generated by pseudoreflections.

Consider then the following element $\bar x\in B$:
\begin{itemize}
\item If $G=G(2,p,n)$ and $n\geq 4$, then $\bar x=(0,a',b',c',x_5,\ldots,x_n)$.\\
Here, $(x_5,\ldots,x_n)\in E^{n-4}$ is a general element and $2a'=a$, $2b'=b$, $2c'=c$, where $E[2]=\{0,a,b,c\}$.
\item If $G=G(2,1,3)$, then $\bar x=(a',b',c')$, where $a',b',c'$ are as above.
\item If $G=G(3,p,n)$, then $\bar x=(0,d,2d,x_4,\ldots,x_n)$.\\
Here $(x_4,\ldots,x_n)\in E^{n-3}$ is a general element and $d\in E[3]$ is not $\zeta_3$-invariant.
\item If $G=G(4,p,n)$, then $\bar x=(0,d,e',x_4,\ldots,x_n)$.\\
Here, $(x_4,\ldots,x_n)\in E^{n-3}$ is a general element, $d$ and $e=2e'$ are in $E[2]$, $d$ is not $\zeta_4$-invariant and $e$ is $\zeta_4$-invariant.
\end{itemize}
The fact that these coordinates are in different $(E_0\times\mu_m)$-orbits is seen as follows. In the first two cases, multiplication by 2 kills the actions of $E_0$ and $\mu_2$ on 4-torsion elements and the coordinates are still all different. In the third case, the action of $\zeta_3$ on $d$ is by translation by a $\zeta_3$-invariant element (say, $e$), so $E_0$ and $\mu_3$ act in the same way on $d$. A direct computation tells us then that $0$, $d$ and $2d$ are in different $(E_0\times\mu_3)$-orbits. In the fourth case, all coordinates have different torsion.

Thus, $\stab_{\Delta\rtimes G}(x)\subset\Delta\times H$ as it was explained above. And easy direct computations in $\Delta\times H$ tell us that the stabilizer of $\bar x$ is given in each case by:
\begin{center}
\begin{tabular}{p{.2\textwidth}|p{.65\textwidth}}
$(m,p,n)$ & Generators of $\stab_{\Delta\rtimes G}(\bar x)\subset \Delta\times H$ \\ \hline
$(2,p,n)$, $n\geq 4$ & $((0,a,b,c,0,\ldots,0),(-1,-1,-1,-1,1,\ldots,1))$ \\
& $(\bar 0,(-1,1,\ldots,1))$ \hfill (exists only if $p=1$).\\ \hline
$(2,1,3)$ & $((a,b,c),(-1,-1,-1))$\\ \hline
$(3,p,n)$ & $((0,2e,e,0,\ldots,0),(\zeta_3,\zeta_3,\zeta_3,1,\ldots,1))$\\
& $(\bar 0,(\zeta_3,1,\ldots,1))$ \hfill (exists only if $p=1$)\\ \hline
$(4,p,n)$ & $((0,e,e,0,\ldots,0),(\zeta_4,\zeta_4,-1,1,\ldots,1))$\\
& $(\bar 0,(-1,1,\ldots,1))$ \hfill (exists only if $p\leq 2$)\\
& $(\bar 0,(\zeta_4,1,\ldots,1))$ \hfill (exists only if $p=1$)
\end{tabular}
\end{center}
In every case, the first element is clearly not a pseudoreflection and it cannot be generated by the others, which proves the proposition.
\end{proof}

The statement of the last proposition hints that the quotient $A/G$ is indeed smooth for $G=G(2,2,3)$. This is actually the case, since it is well-known that $G(2,2,3)$ is isomorphic, as a complex reflection group, to $S_4$ and was therefore already considered in the previous section. The proof of $(1)\Rightarrow(4)$ is now complete.

\subsection{Sporadic groups}\label{sec sporadic}
We deal now with complex reflection groups that are not of the type $G(m,p,n)$. As we recalled before, these are 34 sporadic groups with given actions on $\C^n$ where $n$ varies from 2 to 8.

Let $G$ be such a sporadic group. Recall that having an abelian variety $A$ with an action of $G$ by automorphisms gives us in particular a linear action of $G$ on $T_0(A)\cong\C^n$ that preserves the lattice $\Lambda=\Lambda_A$. We need then some sort of classification of $G$-invariant lattices up to equivalence. A great part of this work was done by Popov in \cite{Popov}, where he studied infinite complex reflection groups, in particular \emph{crystallographic} complex reflection groups, which turn out to be extensions of a finite complex reflection group $G$ by some lattice $\Lambda$ in $\C^n$, where the action of $G$ on $\C^n$ is the one given by Shephard-Todd. In order to deal with sporadic groups, we use then some of Popov's results, which we briefly recall here.\\

First of all, we need the notion of \emph{root lattice}. Given a finite (irreducible) complex reflection group $G$, we can consider the directions on which the pseudoreflections act. With these one can define an actual (irreducible) root system which in turn is useful for classifying these groups (cf.~\cite[\S1]{Popov}). Here, we only care about the lines generated by these roots, that is the eigenspaces of eigenvalue $\neq 1$ for some pseudoreflection $\sigma\in G$, which Popov calls \emph{root lines}. If we consider a $G$-invariant lattice $\Lambda\subset\C^n$, then the sublattices $\Lambda\cap L$ for $L$ a root line generate a $G$-invariant sublattice $\Lambda^0$ of $\Lambda$ called the \emph{root lattice} of $\Lambda$. Note that this is precisely how we constructed the $G$-equivariant isogeny $B\to A$ for $G=G(1,1,n+1)=S_{n+1}$.

We have then the following result, cf.~\cite[\S2.6]{Popov}:

\begin{theorem}[Popov]
The only sporadic groups $G$ in the list of Shephard-Todd that admit a $G$-invariant lattice are the numbers 4, 5, 8, 12, 24--26, 28, 29, 31--37. Their corresponding root lattices are classified up to equivalence by the table in \cite[\S2.6, pp.~37--44]{Popov}.
\end{theorem}

Note that Popov's notion of equivalence of $G$-invariant lattices induces isomorphisms between the corresponding abelian varieties with $G$-action, so that we only need to study abelian varieties $A=\C^n/\Lambda$ for lattices $\Lambda$ such that its root lattice $\Lambda^0$ is in Popov's table. Let us recall then another result that will be useful to classify lattices that are not a root lattice cf.~\cite[\S\S 4.2--4.4]{Popov}.

Consider the endomorphism of $\C^n$ defined as $S:=n\cdot I_n-\sum_{i=1}^n R_i$, where $R_i$ denotes the $i$-th pseudoreflection of a fixed generating set of pseudoreflections of $G$. 

\begin{theorem}[Popov]
Let $\Lambda$ be a $G$-invariant lattice in $\C^n$ and let $\Lambda^0$ be its root lattice. Then $\Lambda^0\subset\Lambda\subset S^{-1}\Lambda^0$. In particular, if $|\det(S)|=1$, then every $G$-invariant lattice is a root lattice.
\end{theorem}

All we are left to do then is to explicitly verify, for each lattice $\Lambda^0$ in Popov's list and for each $G$-invariant lattice between $\Lambda^0$ and $S^{-1}\Lambda^0$, whether the quotient of the corresponding abelian variety by $G$ is smooth or not. As it turns out, this is never true, which we summarize in the following proposition:

\begin{proposition}\label{prop sporadic}
Let $G$ be a sporadic group from the Shephard-Todd list. If $G$ acts on an abelian variety $A$ in such a way that its action on $T_0(A)$ is an irreducible representation, then $A/G$ is not smooth.
\end{proposition}

\begin{proof}
For every such pair $(A,G)$, we consider the associated pair $(\Lambda,G)$, where $A=\C^n/\Lambda$. Tables \ref{table detS=1} and \ref{table detS neq 1} give, for every such pair, a point $x_0\in A$ such that its stabilizer is not generated by pseudoreflections. The result follows then from the Chevalley-Shephard-Todd theorem.\\

We start with the groups $G$ such that $|\det(S)|=1$, so that we only need to verify Popov's explicit lattices. For these, Table \ref{table detS=1} gives:
\begin{itemize}
\item The group $G$ (by giving its number in Shephard-Todd's list).
\item Popov's name for the group $\Lambda^0\rtimes G$.
\item A rational linear combination $v_0$ of the $\Z$-basis $\{e_1,\ldots,e_{2n}\}$ of $\Lambda=\Lambda^0$.
\item The order of the stabilizer $S_0=\stab_G(x_0)$ of the image $x_0$ of $v_0$ in the abelian variety $A=\C^n/\Lambda$.
\item The order of the subgroup $P_0$ of $S_0$ that is generated by pseudoreflections.
\end{itemize}
We refer to \cite[\S2.6, pp.~37--44]{Popov} for the explicit $\Z$-basis. In each case, the first $n$ elements of the basis are Popov's $e_1,\ldots,e_n$ and the $(n+i)$-th element is $\tau_i e_i$ for some explicit $\tau_i\in\C$.

\begin{table}[h!]
\[
\begin{array}{|c|c|c|c|c|}\hline
\# G & \Lambda^0\rtimes G & v_0\in\Lambda^0\otimes_\Z \Q & |S_0| & |P_0| \\[.3em] \hline \hline
5 & [K_{5}] & (\frac13,\frac13,\frac13,0) & 3 & 1 \\[.3em] \hline
8 & [K_{8}] & (\frac13,\frac13,\frac13,0) & 3 & 1 \\[.3em] \hline
12 & [K_{12}] & (0,0,0,\frac12) & 16 & 8 \\[.3em] \hline
24 & [K_{24}] & (\frac14, -\frac14, -\frac14, \frac12, \frac14, -\frac14) & 4 & 1 \\[.3em] \hline
26 & [K_{26}]_1 & (\frac12, \frac12, \frac12, 0, \frac12, \frac12) & 36 & 18 \\[.3em] \hline
26 & [K_{26}]_2 & (0, 0, -\frac13, 0, 0, \frac13) & 72 & 24 \\[.3em] \hline
28 & [F_{4}]_1^\alpha & (\frac12, \frac12, \frac12, \frac12, \frac12, 0, 0, 0) & 12 & 6 \\[.3em] \hline
28 & [F_{4}]_2^\beta & (0,\frac12,\frac12,0,0,0,\frac12,0) & 16 & 8 \\[.3em] \hline
28 & [F_{4}]_3^\gamma & (0,\frac12,0,0,0,0,\frac12,0) & 16 & 8 \\[.3em] \hline
29 & [K_{29}] & (\frac12, 0, 0, 0, \frac12, 0, 0, 0) & 768 & 384 \\[.3em] \hline
31 & [K_{31}] & (\frac12, \frac12, 0, 0, 0, 0, 0, 0) & 384 & 192 \\[.3em] \hline
32 & [K_{32}] & (\frac12, 0, 0, 0, 0, 0, 0, 0) & 1296 & 648 \\[.3em] \hline
34 & [K_{34}] & (\frac13, 0, 0, 0, 0, 0, -\frac13, 0, 0, 0, 0, 0) & 155520 & 51840 \\[.3em] \hline
37 & [E_8]^\alpha & (\frac12, \frac12, \frac12, \frac12, \frac12, \frac12, \frac12, 0, \frac12, \frac12, 0, 0, 0, 0, 0, \frac12) & 103680 & 51840 \\[.3em] \hline
\end{array}
\]
\caption{Examples of non-smooth points in $A/G$ for sporadic groups $G$ such that $|\det(S)|=1$.}\label{table detS=1}
\end{table}

We consider now those groups $G$ in Popov's table for which $|\det(S)|\neq 1$, so that we need to check for new lattices aside from Popov's. These correspond to the numbers 4, 25, 33, 35 and 36 in Shephard-Todd's list. Since we always have $\Lambda^0\subset\Lambda\subset S^{-1}\Lambda^0$ and $[S^{-1}\Lambda^0:\Lambda^0]=|\det(S)|^2$, we see that there are finitely many other lattices to look at. Actually, in all five cases we get that the action of $G$ on the quotient $S^{-1}\Lambda^0/\Lambda^0$ is trivial, so that every lattice in between is a $G$-invariant lattice and needs to be considered. We keep then notations as above (in particular, Popov's $\Z$-basis is given by $\{e_1,\ldots,e_{2n}\}$) and we go case by case:\\

In case 4, a $\Z$-basis for $S^{-1}\Lambda^0$ is given by $\{d_1,d_2,e_3,e_4\}$, where $d_1=\frac12 e_1+\frac12 e_2+\frac12 e_3$ and $d_2=\frac12 e_1+\frac12 e_4$. In particular, we see that the quotient $S^{-1}\Lambda^0/\Lambda^0$ is a Klein group and thus, apart from $S^{-1}\Lambda^0$, we have 3 new lattices to consider: $\Lambda_1:=\langle d_1,\Lambda^0\rangle$, $\Lambda_2:=\langle d_2,\Lambda^0\rangle$, $\Lambda_3:=\langle d_1+d_2,\Lambda^0\rangle$.

In case 25, a $\Z$-basis for $S^{-1}\Lambda^0$ is given by $\{d_1,e_2,\ldots,e_6\}$, where $d_1=\frac13 e_1+\frac13 e_3+\frac23 e_4+\frac23 e_6$. Since the index is 3, this is the only new lattice that needs to be checked.

In case 33, a $\Z$-basis for $S^{-1}\Lambda^0$ is given by $\{d_1,e_2,\ldots,e_5,d_6,e_7\ldots,e_{10}\}$, where $d_1=\frac12 e_1+\frac12 e_3+\frac12 e_5$ and $d_6=\frac12 e_6+\frac12 e_8+\frac12 e_{10}$. In particular, we see that the quotient $S^{-1}\Lambda^0/\Lambda^0$ is a Klein group and thus, apart from $S^{-1}\Lambda^0$, we have 3 new lattices to consider: $\Lambda_1:=\langle d_1,\Lambda^0\rangle$, $\Lambda_2:=\langle d_6,\Lambda^0\rangle$, $\Lambda_3:=\langle d_1+d_6,\Lambda^0\rangle$.

In case 35, a $\Z$-basis for $S^{-1}\Lambda^0$ is given by $\{d_1,e_2,\ldots,e_6,d_7,e_8\ldots,e_{12}\}$, where $d_1=\frac13 e_1-\frac13 e_3+\frac13 e_5-\frac13 e_6$ and $d_7=\frac13 e_7-\frac13 e_{9}+\frac13 e_{11}-\frac13 e_{12}$. In particular, we see that the quotient $S^{-1}\Lambda^0/\Lambda^0$ is isomorphic to $(\Z/3\Z)^2$ and thus, apart from $S^{-1}\Lambda^0$, we have 4 new lattices to consider: $\Lambda_1:=\langle d_1,\Lambda^0\rangle$, $\Lambda_2:=\langle d_7,\Lambda^0\rangle$, $\Lambda_3:=\langle d_1+d_7,\Lambda^0\rangle$, $\Lambda_4:=\langle d_1+2d_7,\Lambda^0\rangle$.

In case 36, a $\Z$-basis for $S^{-1}\Lambda^0$ is given by $\{e_1,d_2,e_3,\ldots,e_8,d_9,e_{10},\ldots,e_{14}\}$, where $d_2=\frac12 e_2+\frac12 e_5+\frac12 e_7$ and $d_9=\frac12 e_9+\frac12 e_{12}+\frac12 e_{14}$. In particular, we see that the quotient $S^{-1}\Lambda^0/\Lambda^0$ is a Klein group and thus, apart from $S^{-1}\Lambda^0$, we have 3 new lattices to consider: $\Lambda_1:=\langle d_2,\Lambda^0\rangle$, $\Lambda_2:=\langle d_9,\Lambda^0\rangle$, $\Lambda_3:=\langle d_2+d_9,\Lambda^0\rangle$.\\

Table \ref{table detS neq 1} gives then, for every pair $(A,G)$ with $A=\C^n/\Lambda$:
\begin{itemize}
\item The group $G$ (by giving its number in Shephard-Todd's list).
\item The corresponding lattice $\Lambda$ (as we named them here above).
\item A rational linear combination $v_0$ of the corresponding $\Z$-basis (as given here above).
\item The order of the stabilizer $S_0=\stab_G(x_0)$ of the image $x_0$ of $v_0$ in the abelian variety $A$.
\item The order of the subgroup $P_0$ of $S_0$ that is generated by pseudoreflections.
\end{itemize}
This concludes the proof of Proposition \ref{prop sporadic}.
\end{proof}

\begin{rem}
For each lattice $\Lambda$ and each ``bad'' element $x_0$ analyzed here above, we computed the stabilizer $S_0$ and its subgroup $P_0$ by brute force using basic SageMath algorithms (we thank once again Antonio Behn for his enormous help in optimizing our first algorithms). Since these are really basic, readers can certainly write their own (and probably in a more efficient manner than ours!). However, for those who would like to look at our code, it is presented in an appendix to a previous version of this article\\
(cf.~\texttt{arxiv.org/abs/1801.00028v2}).

The main idea in order to find these elements was to check the stabilizers (and the pseudoreflections therein) of small torsion elements chosen via the following principle: for every element $g$ of the matrix group $G$, we decomposed $\mathbb{Z}^{2n}$ as $\ker(g-I_{2n})\oplus \ker(g-I_{2n})^\perp$, where the $\perp$ is taken with respect to a $G$-invariant Hermitian form $H$ on $\mathbb{C}^n$. By restricting $g$ to $\ker(g-I_{2n})^\perp$, we obtain an integer-valued matrix $\tilde{g}$ such that $\tilde{g}-I$ is invertible over $\mathbb{Q}$. The columns of $(\tilde{g}-I)^{-1}$ that contain rational, non-integer numbers therefore correspond to fixed points of $g$ in $A$ that do not come from the eigenspace associated to 1 of $g$. These were the vectors whose stabilizers we calculated and analyzed.
\end{rem}

\vfill

\begin{table}[h!]
\[
\begin{array}{|c|c|c|c|c|}\hline
\# G & \Lambda & v_0\in\Lambda\otimes_\Z \Q & |S_0| & |P_0| \\[.3em] \hline \hline
4 & \Lambda^0 & (\frac12, 0, 0, 0) & 2 & 1 \\[.3em] \hline
4 & \Lambda_1 & (0, \frac12, 0, 0) & 4 & 1 \\[.3em] \hline
4 & \Lambda_2 & (0, 0, \frac12, \frac12) & 4 & 1 \\[.3em] \hline
4 & \Lambda_3 & (0, \frac12, \frac12, 0) & 6 & 3 \\[.3em] \hline
4 & S^{-1}\Lambda^0 & (0, 0, 0, \frac12) & 8 & 1 \\[.3em] \hline\hline
25 & \Lambda^0 & (0, -\frac13, 0, 0, \frac13, -\frac13) & 3 & 1 \\[.3em] \hline
25 & S^{-1}\Lambda^0 & (0, 0, 0, \frac13, 0, \frac13) & 72 & 24 \\[.3em] \hline\hline
33 & \Lambda^0 & (0, \frac12, 0, \frac12, \frac12, \frac12, 0, \frac12, 0, \frac12) & 108 & 54 \\[.3em] \hline
33 & \Lambda_1 & (\frac12, 0, 0, 0, 0, 0, 0, 0, 0, 0) & 1296 & 648 \\[.3em] \hline
33 & \Lambda_2 & (0, 0, 0, 0, 0, \frac12, 0, 0, 0, 0) & 1296 & 648 \\[.3em] \hline
33 & \Lambda_3 & (\frac12, 0, 0, 0, 0, \frac12, 0, 0, 0, 0) & 240 & 120 \\[.3em] \hline
33 & S^{-1}\Lambda^0 & (\frac12, 0, 0, 0, 0, \frac12, 0, 0, 0, 0) & 1296 & 648 \\[.3em] \hline\hline
35 & \Lambda^0 & (0, \frac12, 0, \frac12, \frac12, \frac12, 0, \frac12, 0, 0, 0, 0) & 72 & 36 \\[.3em] \hline
35 & \Lambda_1 & (0, \frac13, \frac13, 0, \frac13, 0, 0, 0, 0, 0, 0, 0) & 648 & 216 \\[.3em] \hline
35 & \Lambda_2 & (0, 0, 0, 0, 0, 0, 0, \frac13, \frac13, 0, \frac13, 0) & 648 & 216 \\[.3em] \hline
35 & \Lambda_3 & (0, \frac13, \frac13, 0, \frac13, 0, 0, \frac13, \frac13, 0, \frac13, 0) & 648 & 216 \\[.3em] \hline
35 & \Lambda_4 & (0, \frac13, \frac13, 0, \frac13, 0, 0, -\frac13, -\frac13, 0, -\frac13, 0) & 648 & 216 \\[.3em] \hline
35 & S^{-1}\Lambda^0 & (0, \frac13, \frac13, 0, \frac13, 0, 0, \frac13, \frac13, 0, \frac13, 0) & 648 & 216 \\[.3em] \hline\hline
36 & \Lambda^0 & (\frac12, \frac12, \frac12, \frac12, \frac12, \frac12, \frac12, \frac12, \frac12, 0, 0, 0, 0, \frac12) & 1440 & 720 \\[.3em] \hline
36 & \Lambda_1 & (0, \frac12, 0, 0, 0, 0, 0, 0, 0, 0, 0, 0, 0, 0) & 103680 & 51840 \\[.3em] \hline
36 & \Lambda_2 & (0, 0, 0, 0, 0, 0, 0, 0, \frac12, 0, 0, 0, 0, 0) & 103680 & 51840 \\[.3em] \hline
36 & \Lambda_3 & (0, \frac12, 0, 0, 0, 0, 0, 0, \frac12, 0, 0, 0, 0, 0) & 3840 & 1920 \\[.3em] \hline
36 & S^{-1}\Lambda^0 & (0, \frac12, 0, 0, 0, 0, 0, 0, \frac12, 0, 0, 0, 0, 0) & 103680 & 51840 \\[.3em] \hline
\end{array}
\]
\caption{Non-smooth points in $A/G$ for sporadic groups $G$ such that $|\det(S)|\neq 1$.}\label{table detS neq 1}
\end{table}


\begin{thebibliography}{ABCD00}

\bibitem[Auf17]{Auff} R.~Auffarth. \textit{A note on Galois embeddings of abelian varieties}. Manuscripta Math.~154(3-4), 2017, pp.~279--284.

\bibitem[ALAQ]{Pablo} R.~Auffarth, G.~Lucchini Arteche, P,~Quezada. \textit{Smooth quotients of abelian surfaces by finite groups.} Preprint, 2018.

\bibitem[ALR17]{ALR} R.~Auffarth, H.~Lange, A.~Rojas. \textit{A criterion for an abelian variety to be non-simple}. J.~Pure Appl.~Algebra 221(8), 2017, pp.~1906--1925.

\bibitem[BS06]{BS} J.~Bernstein, O.~Schwarzman. \textit{Chevalley's theorem for the complex crystallographic groups}. J.~Nonlinear Math.~Phys.~13(3), 2006, pp.~323--351.

\bibitem[IL15]{IL} B.~Im, M.~Larsen. \textit{Rational curves on quotients of abelian varieties by finite groups}. Math.~Res.~Lett.~22(4), 2015, pp.~1145--1157.

\bibitem[KL09]{KL} J.~Koll\'ar, M.~Larsen. \textit{Quotients of Calabi-Yau varieties}, in \textit{Algebra, arithmetic, and geometry: in honor of Yu.~I.~Manin. Vol.~II}, Progr.~Math.~270, \textit{Birkhauser Boston Inc., Boston, MA}, 2009, pp.~179--211.

\bibitem[Loo76]{Looijenga} E.~Looijenga. \textit{Root Systems and Elliptic curves}. Invent.~Math.~38(1), 1976, pp.~17--32.

\bibitem[PVdB]{PVdB} A.~Polishchuk, M.~Van den Bergh \textit{Semiorthogonal decompositions of the categories of equivariant coherent sheaves for some reflection groups.} To appear in Journal of the European Mathematical Society.

\bibitem[Pop82]{Popov} V.~Popov. \textit{Discrete complex reflection groups.} Communications of the Mathematical Institute, Rijksuniversiteit Utrecht, 15. \textit{Rijksuniversiteit Utrecht, Mathematical Institute, Utrecht}, 1982, 89 pp.

\bibitem[PZ06]{PopovZ} V.~Popov. Y.~Zarhin. \textit{Finite linear groups, lattices, and products of elliptic curves}. J.~Algebra 305, 2006, pp.~562--576.

\bibitem[ST54]{ST} G.~Shephard, J.~Todd. \textit{Finite unitary reflection groups}. Canadian J.~Math. 6, 1954, pp.~274--304.

\bibitem[TY82]{TY} S.~Tokunaga, M.~Yoshida. \textit{Complex crystallographic groups.~I.} J.~Math.~Soc.~Japan 34(4), 1982, pp.~581--593.

\bibitem[Yos07]{Yoshi} H.~Yoshihara. \textit{Galois embedding of algebraic variety and its application to abelian surface}. Rend.~Semin.~Mat.~Univ.~Padova 117, 2007, pp.~69-85. 

\end{thebibliography}
\end{document}